\documentclass[12pt]{article}
\usepackage[utf8]{inputenc}
\usepackage[T2A]{fontenc}
\usepackage{cmap}
\usepackage{indentfirst}
\usepackage[english]{babel}
\usepackage{amsmath,amssymb,amsthm,graphicx}
\usepackage{bbm}
\usepackage[usenames]{color}
\usepackage{colortbl}
\usepackage{caption}
\usepackage{subcaption}
\usepackage[unicode, pdftex]{hyperref}
\usepackage[left=2.5cm,right=2.2cm,top=2.5cm,bottom=3.5cm]{geometry}
\usepackage{wasysym}

\newtheorem{theorem}{Theorem}[section]

\newtheorem{lemma}{Lemma}[section]
\newtheorem{conjecture}{Conjecture}[section]%

\newcommand{\keywords}[1]{\textbf{Keywords:} #1}
\newcommand{\msc}[1]{\textbf{MSC Classification:} #1}

\begin{document}
\title{Mahler-type volume inequality for convex bodies with tetrahedral symmetry}

\author{Arkadiy Aliev\footnote{HSE University, Moscow, Russia, \texttt{\href{mailto:arkadiy.aliev@gmail.com}{arkadiy.aliev@gmail.com}}}}
\date{}
\maketitle

\begin{abstract}
Let \( K \) be a convex body in \( \mathbb{R}^n \). We denote the volume of \( K \) by \( |K| \), and the polar body of its difference body \( K - K \) by \( (K - K)^{\circ} \). We provide a new proof of the well-known estimate
\[
|K||(K - K)^{\circ}| \geq \frac{3}{2}
\]
for \( K \subset \mathbb{R}^2 \), with equality attained for a triangle. For \( K \subset \mathbb{R}^3 \) with tetrahedral symmetry, we prove that
\[
|K| |(K - K)^{\circ}| \geq \frac{2}{3},
\]
with equality attained for a tetrahedron.

\end{abstract}


 \msc{52A10, 52A15, 52A38, 52A40}

\section{Introduction}
     We call a convex compact set $K\subset\mathbb{R}^{n}$ with non-empty interior $\mathrm{int}K$ a convex body. The volume of $K$ is denoted by $\vert K \vert$. The difference body of $K$ is defined by
     
     $$
        K-K = \{x-y\text{ }\vert\text{ }x,y\in K\}.
     $$ 
 We say that $K$ is centrally symmetric if $K=-K$. We denote the family of all convex bodies in $\mathbb{R}^{n}$ by $\mathcal{K}^{n}$ and the family of all centrally symmetric convex bodies in $\mathbb{R}^{n}$ by $\mathcal{K}_{0}^{n}$. The polar body of $K$ is defined as
    $$
        K^{\circ}:=\{x\in\mathbb{R}^{n} \text{ }\vert\text{ } \forall y\in K\text{ } \langle y,x \rangle\leq 1 \}.
    $$
The volume product of $K \in \mathcal{K}^{n}$ is defined as 
$$\mathcal{P}(K) = \min_{z\in K} \vert K\vert \vert (K - z)^{\circ}\vert.$$ The well known Blaschke--Santal\'o inequality \cite[p.548]{bruhn} gives the upper bound $\mathcal{P}(K)\leq \mathcal{P}(B_{2}^{n})$, where we denote $B_{p}^{n} = \{x\in \mathbb{R}^{n}:\| x\|_{p}\leq 1\}$. The lower bound for $K\in \mathcal{K}^{n}_{0}$ is the famous Mahler's conjecture.

    \begin{conjecture}[Symmetric Mahler’s conjecture]
        \label{mahler}
        For a convex body $K\in \mathcal{K}^{n}_{0}$.
        $$
        \mathcal{P}(K) \geq  \mathcal{P}(B_{\infty}^{n}) = \frac{4^{n}}{n!}.
        $$
    \end{conjecture}
    It was fully resolved for $n=2$ \cite{ma1} and $n=3$ \cite{mahler3d}. In the case of general dimension, the conjecture has been proven for some families of convex bodies including unconditional convex boides \cite{symmahl}, zonoids \cite{zonoids}, for bodies sufficiently close to the unit cube in the Banach-Mazur distance \cite{petrov}. The best known lower bound is $\mathcal{P}(K)\geq \frac{\pi^{n}}{n!}$ due to Kuperberg \cite{kup}.

    Symmetric Mahler's conjecture plays a central role in the geometry of numbers and estimating of densities of non-separable lattice arrangements. We say that $\Lambda\subset \mathbb{R}^{n}$ is a lattice if $\Lambda = A\mathbb{Z}^n$ for some $A\in GL(n,\mathbb{R})$. We define $\det(\Lambda) = |det A|.$ A set of translates of $K$ by a lattice is defined as $\Lambda + K = \{x+y \text{ } \vert\text{ } x\in \Lambda,y\in K\}$ and its density is defined as $d(\Lambda, K) := \frac{\vert K\vert}{d(\Lambda)}.$  

    We say that $\Lambda +K$ is non-separable if each affine $(n-1)$-subspace in $\mathbb{R}^{n}$ meets $x+K$ for some $x\in\Lambda$. Non-separable set of translates by a lattice were introduced by  L. Fejes T\'oth in \cite{Mak-1974}, see also \cite{Mak-1978, Mak2016}. For a convex body $K\subset \mathbb{R}^{n}$ we define 
        $$d_{n}(K) := \min\{d(\Lambda,K)\text{ }\vert\text{ } \Lambda+K \text{ is non-separable} \}.$$ 
    The lower bounds for $d_{n}(K)$ (see \cite{aliev, Mak-1978, Mak2016}) are the polar dual results to the famous Minkowski’s fundamental theorem (see \cite{shimura} for the explanation of the duality), which states that any centrally symmetric convex body $K\in \mathcal{K}^{n}_{0}$ without nonzero integer points in its interior has volume non greater than $2^{n}$. It was shown in \cite{Mak-1978} that $d_{n}(K)\geq \frac{1}{4^{n}}\mathcal{P}(K)$ for $K\in \mathcal{K}^{n}_{0}$. Thus, $d_{2}(K)\geq \frac{1}{2}$ and $d_{3}(K)\geq \frac{1}{6}$. Moreover, if $K$ is a cross-polytope, then $d_{n}(K)=\frac{1}{n!}$, which shows that the Mahler conjecture implies tight lower bounds for $d_{n}$ in the centrally symmetric case. The classical non-symmetric extension of the Mahler conjecture (see \cite{ma2, mahlernonsym}) however does not imply tight lower bounds for $d_{n}$ in the general case.

    \begin{conjecture}[Non-symmetric Mahler’s conjecture]
        \label{mahler_nonsym}
        For a convex body $K\in \mathcal{K}^{n}$.
        $$
        \mathcal{P}(K) \geq  \mathcal{P}(\Delta_{n}) =\frac{(n+1)^{n+1}}{(n!)^{2}},
        $$
        where $\Delta_{n} \in \mathcal{K}^{n}$ is a regular simplex.
    \end{conjecture}

      It was shown in \cite{Mak-1978} that  $d_{n}(K)\geq \frac{1}{2^{n}}\vert K\vert \vert (K-K)^{\circ}\vert$. This leads us to the alternative non-symmetric extension of the Mahler conjecture that preserves its connection to the geometry of numbers and non-separable lattices.
    
    \begin{conjecture}[Makai Jr \cite{Mak-1978}]
    \label{main-conj}
        For a convex body $K\subset \mathbb{R}^{n}$.
        $$
        \vert K\vert \vert(K-K)^{\circ}\vert \geq \frac{n+1}{n!}.
        $$
        The equality holds if and only if $K$ is a simplex.

    \end{conjecture}

    The two-dimensional case of this conjecture was initially discovered by Eggleston~\cite{Egg} and thus the tight bound $d_{2}(K)\geq \frac{3}{8}$ holds with equality attained for a triangle. The idea of Eggleston is to construct, for a given convex body, a triangle of the same area and greater width in every direction. Another proof, based on restricted chord projections, was later found by Zang~\cite{zang}. Moreover, stability for the two-dimensional case was established in~\cite{volprodstability}.  We give a new proof of Conjecture~\ref{main-conj} for \( n = 2 \) using affine regular hexagons, that is, non-degenerate affine images of the regular hexagon.

    \begin{theorem}[Eggleston, \cite{Egg}]
     \label{plane}
        For a convex body $K\subset \mathbb{R}^{2}$.
            $$
                \vert K\vert \vert(K-K)^{\circ}\vert \geq \frac{3}{2}.
            $$
        The equality holds if and only if $K$ is a triangle.

    \end{theorem}

    In the general dimension, the best known estimates \cite{Mak2016} are obtained by combining the bounds from the symmetric Mahler conjecture and the Rogers-Shephard inequality \cite{diff}.

    $$
        \vert K \vert \vert (K-K)^{\circ}\vert = \frac{ \vert K \vert }{\vert (K-K)\vert}\vert (K-K)\vert\vert (K-K)^{\circ}\vert \geq \binom{2n}{n}^{-1}\mathcal{P}(K-K) \geq  \binom{2n}{n}^{-1}\frac{\pi^{n}}{n!}.
    $$
    Note that for $n=3$ this approach gives $\vert K \vert \vert (K-K)^{\circ}\vert\geq \frac{8}{15}$ which is very close to the conjectured $\frac{2}{3}$.

    For a convex body $K\subset \mathbb{R}^{3}$ we define
    $SO(K):=\{g\in SO_{3}\text{ }\vert\text{ }gK=K\}.$ We say that a convex body $K\in \mathcal{K}^{3}$ has \textit{tetrahedral symmetry}, if $SO(\Delta_{3})\subset SO(K),$ where 
        $$
            \Delta_{3} = conv \{(1,1,-1),(1,-1,1),(-1,1,1),(-1,-1,-1)\}.
        $$
    
  We denote by $\mathcal{K}^{3}_{\Delta}$ the family of all three-dimensional convex bodies with tetrahedral symmetry. Using ideas similar to \cite{mahler3d} and \cite{mahlersym}, we prove Conjecture \ref{main-conj} for $K\in \mathcal{K}^{3}_{\Delta}$. This result also provides an alternative proof of the recent bound $d_{3}(K)\geq \frac{1}{12}$ \cite{aliev} for such bodies.
    
    \begin{theorem}
    \label{space}
        For a convex body $K \in \mathcal{K}^{3}_{\Delta}$. 
        $$
            \vert K\vert \vert(K-K)^{\circ}\vert\geq \frac{2}{3}.
        $$
        Among convex bodies possessing tetrahedral symmetry, equality holds if and only if $K$ is a tetrahedron.
    \end{theorem}

\section{New proof for $n=2$}

    For $u\in \mathbb{R}^{n}$ we denote the orthogonal projection on $u^{\perp} = \{x\in\mathbb{R}^{n} \text{ }: \text{ } \langle x,u\rangle = 0\}$ by $Pr_{u^{\perp}}$. 
    \begin{lemma}[Zang \cite{zang}]
        \label{zang_lemma}
        For $K\in \mathcal{K}^{n}$ and any $u \in \partial(K-K)$.  
        $$
            \vert K\vert\geq \frac{1}{n}\vert u\vert \vert Pr_{u^{\perp}}K \vert.
        $$
    \end{lemma}

    \begin{proof}
  Let \( u = a - b \) for \( a, b \in \partial K \). We denote by \(\tilde{K}\) the Steiner symmetrization of \(K\) with respect to \(u^{\perp}\), and let \(\tilde{a}\) and \(\tilde{b}\) be the images of \(a\) and \(b\) under this symmetrization. 

        $$
            \vert K\vert = \vert \tilde{K}\vert \geq \vert\mathrm{conv}\{\tilde{a}, \tilde{b}, Pr_{u^{\perp}}K\}\vert = \frac{1}{n}\vert u\vert \vert Pr_{u^{\perp}}K \vert.
        $$
    
    \end{proof}

        Note that in the planar case Lemma \ref{zang_lemma} implies the following statement. If four points \( p_1, p_2, p_3, p_4 \in K \subset \mathbb{R}^{2} \) satisfy
\[
|K| = \frac{1}{2} \left| \det(p_3 - p_1,\, p_4 - p_2) \right|,
\]
then \( K = \mathrm{conv}(p_1, p_2, p_3, p_4) \), and moreover, $K$ is a triangle or $K$ is a polygon and these points are the vertices of \( K \) in cyclic order.
    
    \begin{proof}[Proof of Theorem~\ref{plane}]
    
        \begin{figure}
            \centering
            \begin{subfigure}{.5\textwidth}
                \centering
                \includegraphics[width=.7\linewidth]{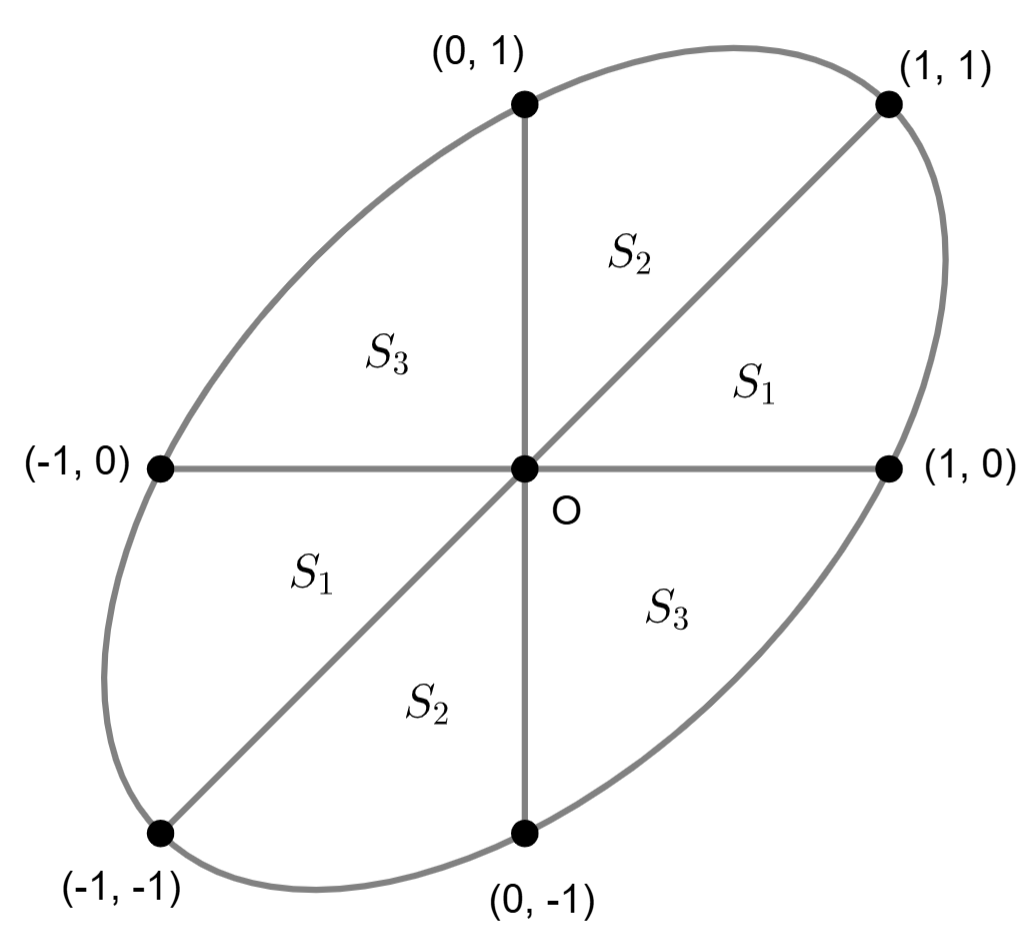}
                \caption{Partition of $(K-K)^{\circ}$}
                \label{fig:1a}
            \end{subfigure}%
            \begin{subfigure}{.5\textwidth}
                \centering
                \includegraphics[width=.7\linewidth]{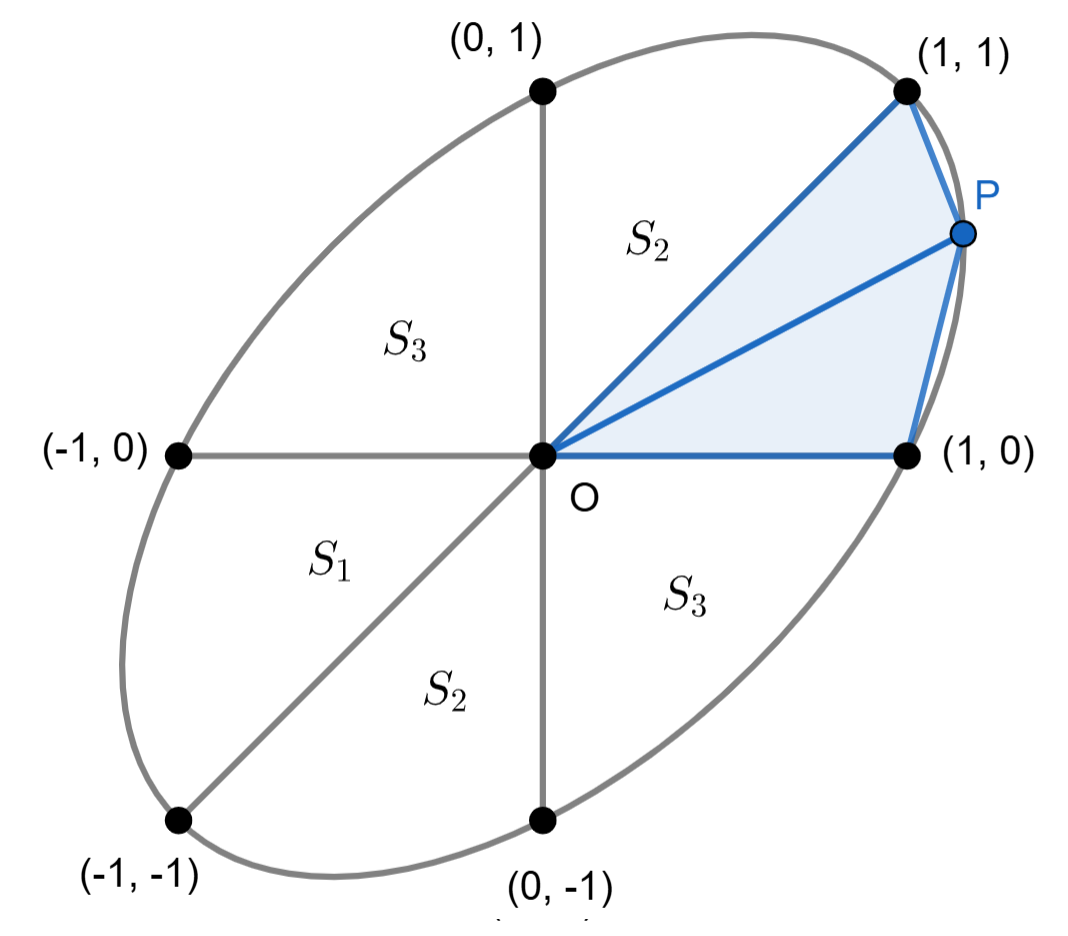}
                \caption{Area estimation}
                \label{fig:1b}
            \end{subfigure}
            \label{fig:pred}
            \caption{Proof idea for the planar case}
        \end{figure}

        It was proved by Lassak \cite{lassak} that for any centrally symmetric convex body \(L \in \mathcal{K}^{2}_{0}\) and any direction, there exists an affine regular hexagon \(A_1 \ldots A_6\) inscribed in \(L\) such that \(A_1 A_2\) is parallel to this direction. Thus, after an appropriate linear transformation, we can assume that \(\pm(1,0), \pm(0,1), \pm(1,1) \in \partial (K-K)^{\circ}\). We split $(K-K)^{\circ}$ into $6$ pieces by the lines $(0,1)^{\perp
        }, (1, 0)^{
        \perp
        }, (1,-1)^{\perp
        }$ as in the Figure~\ref{fig:1a}. Then for any $P=(u,v) \in (K-K)^{\circ}$ we can estimate the area $S_{1}$ from below by the sum of signer areas of $\Delta O(0,1)P$  and $\Delta OP(1, 1)$ as in the Figure~\ref{fig:1b}. 
    
    $$
        S_{1} \geq \frac{1}{2}\det((1,0),(u,v)) + \frac{1}{2}\det((u,v),(1,1)) = \frac{1}{2}\langle(u,v),(1,0)\rangle  \implies \frac{1}{2S_{1}}(1,0) \in (K-K).
    $$
    Similar argument for $S_{2}$ and $S_{3}$ gives 
    $\frac{1}{2S_{2}}(0,1)\in K-K$ and $ \frac{1}{2S_{3}}(-1,1) \in K-K.$ 

        Since $(0,1), (1,0), (1,1) \in \partial(K-K)^{\circ}$, exist $(a,1), (1,b), (c, 1-c) \in \partial(K-K) $. Therefore, by Lemma~\ref{zang_lemma}, we get the following inequalities

    $$
        \begin{cases}
            \vert K\vert\geq \frac{1}{2}\vert det((a,1), \frac{1}{2S_{1}}(1,0))\vert  =
            \frac{1}{4S_{1}}; \\
            \vert K\vert \geq \frac{1}{2}\vert det((1,b), \frac{1}{2S_{2}}(0,1))\vert  = \frac{1}{4S_{2}};\\
              \vert K\vert\geq \frac{1}{2}\vert det((c,1-c), \frac{1}{2S_{3}}(-1,1))\vert  = \frac{1}{4S_{3}}.
        \end{cases}
    $$
    Combining these estimates, we conclude that 
    $$
        \vert K \vert \vert (K-K)^{\circ}\vert  = 2\vert K\vert S_{1} + 2\vert K\vert S_{2} + 2\vert K\vert S_{3} \geq 6\vert K\vert\min(S_{1},S_{2},S_{3}) \geq\frac{3}{2}.
    $$
    \end{proof}
    
    \begin{proof}[Equality case.]
        Equality can only be achieved if $S_{1}=S_{2}=S_{3}=S$ and $\frac{1}{2S}(0,1)$, $\frac{1}{2S}(1,0)$, $\frac{1}{2S}(-1,1)\in \partial (K-K).$ In addition, the necessary condition is that
        $$\vert K\vert = \frac{1}{4S} = \begin{cases}
            \frac{1}{2}\det((1,b), \frac{1}{2S}(0,1)); \\ 
            \frac{1}{2}\det(\frac{1}{2S}(1,0), (a,1)); \\
            \frac{1}{2}\det(\frac{1}{2S}(1,-1), (c,1-c)).
        \end{cases} $$
We use the implication of Lemma \ref{zang_lemma} for the planar case. Assume that $K$ is not a triangle and that $K = \mathrm{conv}\{0, p_{1}, \frac{1}{2S}(0,1), p_{2}\}$ with $p_{2}-p_{1} = (1, b)$. Using the second equality we get $p_{2}-p_{1} = \frac{1}{2S}(1,0)$ and thus $b=0$ and $S=\frac{1}{2}.$ Using the third equality we get $(0,1) = (-1,1)$ or $(1,0) = (-1,1)$ thus a contradiction. 
    \end{proof}

\section{Special case for $n=3$.}

    \begin{proof}[Proof of Theorem~\ref{space}]     
    We split $(K-K)^{\circ}$ into $14$ pieces by the planes $(1,1,1)^{\perp}$, $ (1,-1,1)^{\perp}$, $(-1,1,1)^{\perp}$, $(-1,-1,1)^{\perp}$. By symmetry, there are $8$ blue pieces congruent to $(K-K)^{\circ}\cap \{-x+ y+ z\geq 0, x-y+z\geq 0, x+y-z\geq 0\}$ and $6$  red pieces congruent to $(K-K)^{\circ}\cap \{ \pm x \pm y + z\geq 0\}$ (see Figure  \ref{fig:sub1}). We denote the volume of each blue piece as $\frac{1}{8}V_{1}$ and the volume of each red piece as $\frac{1}{6}V_{2}.$ Clearly, $V_{1} + V_{2} = \vert (K-K)^{\circ}\vert.$ 
    
    We use special sections of $(K-K)^{
    \circ
    }$ (see Figures \ref{fig:sub2} and \ref{fig:sub3}) to estimate $\vert Pr_{(0,0,1)^{\perp}}K\vert $ and $\vert Pr_{(1,1,1)^{\perp}} K\vert$ from below. We define 
    $$
    \begin{cases}
        S_{\hexagon} = \frac{1}{6}\vert (K-K)^{\circ}\cap (1,1,1)^{\perp}\vert; \\
        S_{\square} = \frac{1}{4}|(K-K)^{\circ}\cap (0,0,1)^{\perp}|.
    \end{cases}
    $$
     The projection $Pr_{(0,0,1)^{\perp}}K$ is centrally symmetric since  $(x,y,z)\mapsto(-x,-y,z) \in SO(\Delta_{3}).$ It is well-known that $(K-K)^{\circ}\cap u^{\perp} = (Pr_{u^{\perp}}(K-K))^{\circ}$, where the polar of $Pr_{u^{\perp}}(K-K)$ is taken in $u^{\perp}$ (e.g. \cite[p.22]{geomtom}). Therefore, by the planar Mahler's conjecture  we have 

     $$
        \vert Pr_{(0,0,1)^{\perp}}K\vert \vert (K-K)^{\circ} \cap (0,0,1)^{\perp}\vert 
 = \frac{1}{4}\vert Pr_{(0,0,1)^{\perp}}K -Pr_{(0,0,1)^{\perp}}K \vert \vert (K-K)^{\circ}\cap (0,0,1)^{\perp}\vert = 
     $$
     \begin{equation}
        = \frac{1}{4}\vert Pr_{(0,0,1)^{\perp}}(K-K)  \vert \vert (K-K)^{\circ}\cap (0,0,1)^{\perp}\vert\geq 2 \implies \vert Pr_{(0,0,1)^{\perp}}K\vert  \geq \frac{1}{2S_{\square}}.
        \label{square_est}
        \tag{$\square$}
      \end{equation}
    By Theorem \ref{plane} we have
     \begin{equation}
        \vert Pr_{(1,1,1)^{\perp}} K\vert \vert (K-K)^{\circ}\cap (1,1,1)^{\perp}\vert \geq \frac{3}{2} \implies \vert Pr_{(1,1,1)^{\perp}} K\vert \geq \frac{1}{4S_{\hexagon}}.
        \label{hexagon_est}
        \tag{$\hexagon$}
     \end{equation}
     
   Next, we establish lower bounds for $V_{1}\vert K\vert$ and $V_{2}\vert K\vert$ in terms of $S_{\square}$ and $S_{\hexagon}$. 
    
 \begin{figure}
\centering
\begin{subfigure}{.33\textwidth}
  \centering
  \includegraphics[width=\linewidth]{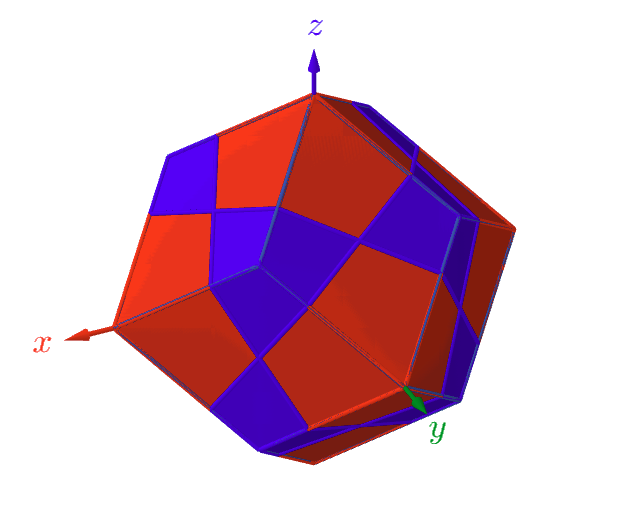}
  \caption{Partition of $(K-K)^{\circ}$}
  \label{fig:sub1}
\end{subfigure}%
\begin{subfigure}{.33\textwidth}
  \centering
  \includegraphics[width=\linewidth]{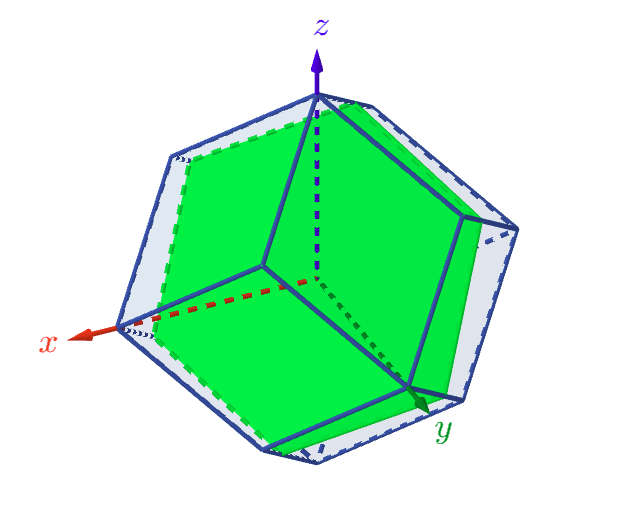}
  \caption{$6S_{
  \hexagon
  }$}
  \label{fig:sub2}
\end{subfigure}
\begin{subfigure}{.33\textwidth}
  \centering
  \includegraphics[width=\linewidth]{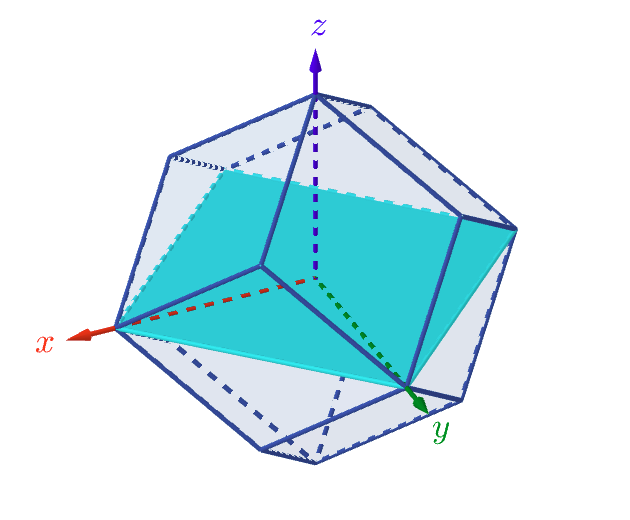}
  \caption{$4S_{
  \square
  }$}
  \label{fig:sub3}
\end{subfigure}

\label{fig:test2}
\caption{Partition and sections used in the proof }
\end{figure}

%     \begin{figure}
% \centering
% \begin{subfigure}{.5\textwidth}
%   \centering
%   \includegraphics[width=.8\linewidth]{44.png}
%   \caption{$4S_{\square}$}
%   \label{fig:sub3}
% \end{subfigure}%
% \begin{subfigure}{.5\textwidth}
%   \centering
%   \includegraphics[width=.8\linewidth]{66.png}
%   \caption{$6S_{\hexagon}$}
%   \label{fig:sub4}
% \end{subfigure}

% \label{fig:test1}
% \caption{Central sections used in the proof.}
% \end{figure}

 \begin{figure}
\centering
\begin{subfigure}{.33\textwidth}
  \centering
  \includegraphics[width=\linewidth]{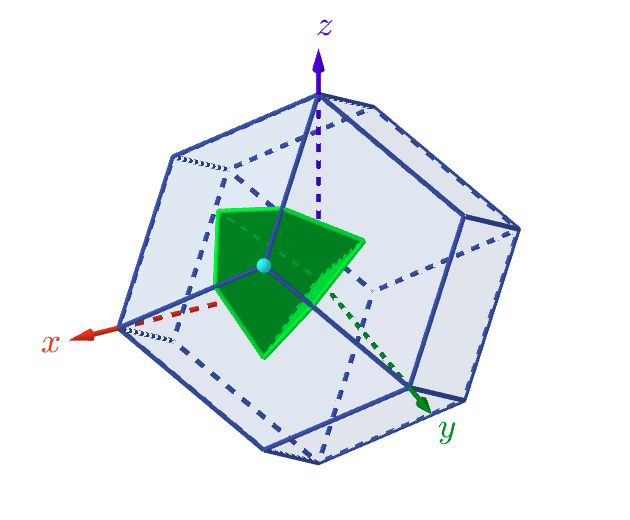}
  \caption{$a(1,1,1)$ and $A_
{
\hexagon
}^{(1,2,3)}$}
  \label{fig:sub5}
\end{subfigure}%
\begin{subfigure}{.33\textwidth}
  \centering
  \includegraphics[width=\linewidth]{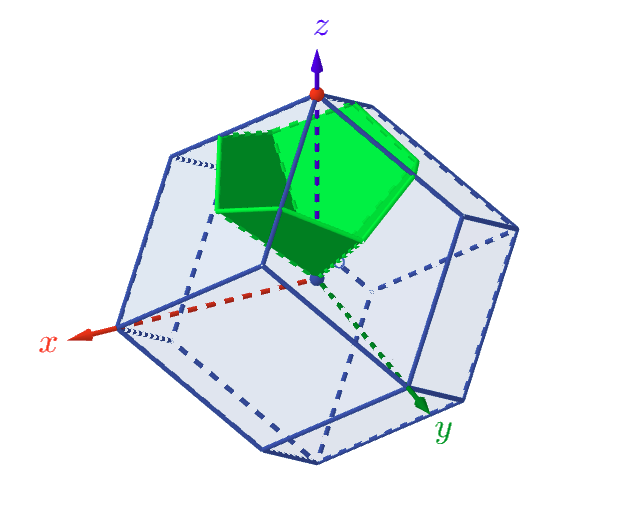}
  \caption{$b(0,0,1)$ and $A^{(\pm\pm
  )}_{\hexagon}$}
  \label{fig:sub6}
\end{subfigure}
\begin{subfigure}{.33\textwidth}
  \centering
  \includegraphics[width=\linewidth]{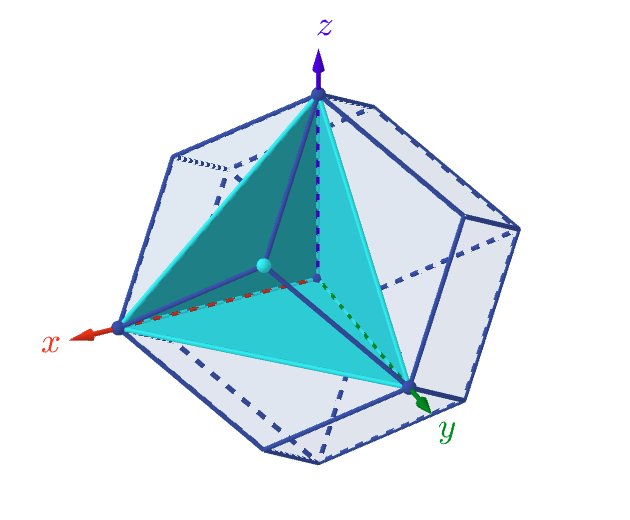}
  \caption{$a(1,1,1)$ and $A_{\square}^{(1,2,3)}$}
  \label{fig:sub7}
\end{subfigure}

\label{fig:test2}
\caption{Partition and sections used to get Estimates
\ref{Est1}, 
\ref{Est2}, \ref{Est3} }
\end{figure}

    \noindent\textbf{Estimate 1. }Let $a(1,1,1) \in \partial (K-K)^{\circ}.$ By symmetry,  $\frac{1}{3a}(1,1,1) \in \partial (K-K).$ We denote $$
    \begin{cases}
     A^{(1)}_{\hexagon} = (K-K)^{\circ}\cap\{-x+y+z = 0, x - y + z\geq 0, x+y-z\geq 0\}; \\
     
     A^{(2)}_{\hexagon} = (K-K)^{\circ}\cap\{- y + z+x=0 ,  y-z+x\geq 0, y+z-x \geq 0 \}; \\

     A^{(3)}_{\hexagon} = (K-K)^{\circ}\cap\{-z+x+y  = 0, z - x + y\geq 0, z+x-y\geq 0\}. 
    \end{cases}
    $$
    Clearly, $\vert A^{(1)}_{\hexagon} \vert = \vert A^{(2)}_{\hexagon} \vert =\vert A^{(3)}_{\hexagon} \vert =S_{\hexagon}.$ We have the following chain of inequalities.

    $$
        \frac{1}{8}V_{1}\geq \vert \mathrm{conv}\{a(1,1,1), A^{(1)}_{\hexagon}, A^{(2)}_{\hexagon},A^{(3)}_{\hexagon}\} \vert \geq 3\frac{1}{3}S_{\hexagon}\vert \langle a(1,1,1) ,\frac{1}{\sqrt{3}}(1,1,-1)\rangle\vert = \frac{a S_{\hexagon}}{\sqrt{3}}.
    $$
    By Lemma~\ref{zang_lemma} we have 
        \begin{equation}
        \begin{split}
        \vert K \vert V_{1} \geq \frac{1}{3}\vert \frac{1}{3a} (1,1,1)\vert \vert Pr_{(1,1,1)^{\perp}}K\vert V_{1} = \frac{1}{3\sqrt{3}a} \vert Pr_{(1,1,1)^{\perp}}K\vert V_{1}\geq  \\ \geq  \frac{1}{3\sqrt{3}a} \frac{1}{4S_{\hexagon}}V_{1} =
        \frac{1}{12\sqrt{3}aS_{\hexagon}} V_{1} \geq \frac{1}{12\sqrt{3}aS_{\hexagon}}8\frac{a S_{\hexagon}}{\sqrt{3}} = \frac{2}{9}. 
         \label{Est1}
         \end{split}
    \end{equation}
    
    \noindent\textbf{Estimate 2. }Let $b(0,0,1) \in \partial (K-K)^{\circ}.$ By symmetry,  $\frac{1}{b}(0,0,1) \in \partial (K-K).$ We denote $$
    \begin{cases}
     A^{(--)}_{\hexagon} = (K-K)^{\circ}\cap\{-x + y + z\geq 0,x+y+z = 0,  x-y+z\geq 0\}; \\
      A^{(+-)}_{\hexagon} = (K-K)^{\circ}\cap\{x+y+z \geq 0, -x + y + z=0 , -x-y+z\geq 0\}; \\
     A^{(++)}_{\hexagon} = (K-K)^{\circ}\cap\{x-y+z  \geq 0, -x - y + z=0, -x+y+z\geq 0\}; \\
      A^{(-+)}_{\hexagon} = (K-K)^{\circ}\cap\{-x-y+z  \geq 0, x - y + z=  0, x+y+z\geq 0\}. \\
    \end{cases}.
    $$
    We have the following chain of inequalities. 
    $$
        \frac{1}{6}V_{2} \geq \vert \mathrm{conv}\{b(0,0,1), A^{(--)}_{\hexagon}, A^{(-+)}_{\hexagon},A^{(+-)}_{\hexagon},A^{(++)}_{\hexagon}\}\vert \geq $$
    $$\geq 4\frac{1}{3}S_{\hexagon} \vert \langle b(0,0,1),\frac{1}{\sqrt{3}}(1,1,1)\rangle\vert = \frac{4}{3\sqrt{3}}S_{\hexagon}b.
    $$
By Lemma~\ref{zang_lemma} we have

\begin{equation}
\begin{split} 
        \vert K \vert V_{2} \geq \frac{1}{3} \vert \frac{1}{b}(0,0,1)\vert \vert Pr_{(0,0,1)^{\perp}}K\vert V_{2} = \frac{1}{3b} \vert Pr_{(0,0,1)^{\perp}}K\vert V_{2} \geq \\ \geq \frac{1}{3b}\frac{1}{2S_{\square}} V_{2} = \frac{1}{6b S_{\square}} V_{2}\geq 
\frac{1}{6b S_{\square}}6\frac{4}{3\sqrt{3}}S_{\hexagon}b = \frac{4}{3\sqrt{3}}\frac{S_{\hexagon}}{S_{\square}}.
\end{split}
   \label{Est2}
\end{equation}

    \noindent\textbf{Estimate 3.} 
Let us recall what we assumed $a(1,1,1) \in \partial (K-K)^{\circ}$. We denote $$
    \begin{cases}
     A^{(1)}_{\square} = (K-K)^{\circ}\cap\{x=0, y\geq 0, z\geq 0\}; \\
      A^{(2)}_{\square} = (K-K)^{\circ}\cap\{x\geq 0, y = 0, z\geq 0\}; \\

     A^{(3)}_{\square} = (K-K)^{\circ}\cap\{x\geq 0, y\geq 0, z= 0\}. 
    \end{cases}
    $$
    Clearly, $\vert A^{(1)}_{\square} \vert = \vert A^{(2)}_{\square} \vert =\vert A^{(3)}_{\square} \vert =S_{\square}.$ We have the following chain of inequalities.

    $$\frac{1}{8}\vert (K-K)^{\circ}\vert \geq \vert \mathrm{conv}\{a(1,1,1), A^{(1)}_{\square}, A^{(2)}_{\square}, A^{(3)}_{\square}\}\vert
     \geq 3\frac{1}{3}S_{\square} \vert \langle a(1,1,1),(0,0,1)\rangle\vert = aS_{\square}.
    $$
    Then, as in Estimate~\ref{Est1}, we get 
    \begin{equation}
        \vert K \vert \vert (K-K)^{\circ}\vert \geq \frac{1}{12\sqrt{3}aS_{\hexagon}} \vert (K-K)^{\circ}\vert \geq  \frac{1}{12\sqrt{3}aS_{\hexagon}} 8aS_{\square}  \geq  \frac{2}{3\sqrt{3}} \frac{S_{\square}}{S_{\hexagon}}.
        \label{Est3}
    \end{equation}

    To finish the proof we consider three cases.
    
    \noindent\textbf{Case $V_{2}\geq 2V_{1}$.}  Estimate~\ref{Est1} implies  
    $$
      \vert K \vert \vert (K-K)^{\circ}\vert = \vert K \vert (V_{1} + V_{2}) \geq 3 \vert K \vert V_{1} \geq 3\frac{2}{9} = \frac{2}{3}.  
    $$

    \noindent\textbf{Case } $2V_{1}\geq V_{2}$ and $\frac{S_{\hexagon}}{S_{\square}}\geq \frac{1}{\sqrt{3}}$\textbf{.} By Estimate~\ref{Est2} we get 
    $$
      \vert K \vert \vert (K-K)^{\circ}\vert = \vert K \vert (V_{1} + V_{2}) \geq \frac{3}{2} \vert K \vert V_{2} \geq \frac{3}{2} \frac{4}{3\sqrt{3}}\frac{S_{\hexagon}}{S_{\square}} \geq \frac{3}{2} \frac{4}{3\sqrt{3}} \frac{1}{\sqrt{3}} = \frac{2}{3}.  
    $$

    \noindent\textbf{Case} $\frac{S_{\square}}{S_{\hexagon}}\geq \sqrt{3}$\textbf{.} In this case we use Estimates~\ref{Est1}, \ref{Est2}, \ref{Est3}.
    $$
         \vert K \vert \vert (K-K)^{\circ}\vert = \frac{1}{2}\left(\vert K \vert V_{1} + \vert K \vert V_{2} + \vert K \vert \vert (K-K)^{\circ} \right) \geq 
    $$
    $$
        \geq \frac{1}{2}\left(\frac{2}{9} + \frac{4}{3\sqrt{3}}\frac{S_{\hexagon}}{S_{\square}} 
 +\frac{2}{3\sqrt{3}} \frac{S_{\square}}{S_{\hexagon}} \right) \geq \frac{2}{3} \text{ for }\frac{S_{\square}}{S_{\hexagon}}\geq \sqrt{3}.
    $$
    % \url{https://www.wolframalpha.com/input?i=min+1%2F2%282%2F9+%2B+2%2F%283sqrt%283%29%29x+%2B+4%2F%283sqrt%283%29x%29%29+for+x+%3E%3D+%28sqrt%283%29%29+}
    % \end{enumerate}
        
    \end{proof}

    \begin{lemma}
        \label{56points}
        Assume that a convex polyhedron $P$ with tetrahedral symmetry has at most $6$ vertices. Then $P$ is a tetrahedron or an octahedron. 
    \end{lemma}

    \begin{proof}
        Let $V$ be the set of vertices of $P$. Let $g:(x,y,z)\mapsto (y,z,x)$, $g\in SO(\Delta_{3})$. We consider the orbits of the element $g$ under its action on the set $V$. The length of each orbit is either $3$ or $1$. The number of orbits of length one cannot be greater than two, because in that case, one of the vertices of $P$ would lie on the segment between the other two. Therefore there is at least one orbit of length $3$. Let $(p_{1},p_{2},p_{3})$, $(p_{2},p_{3},p_{1})$, $(p_{3},p_{1},p_{2}) \in V$. If $\vert V\vert=4$ then $P$ is a tetrahedron.

        \textbf{Case} $\vert V\vert=5$\textbf{.} Then there are two orbits of length one. Let $(p,p,p),(q,q,q)\in V$ and $p
        \neq 0$. Then $(p,-p,-p), (-p,p,-p),(-p,-p,p)\in V$. Then $q \neq 0$ and $(q,-q,-q)$, $(-q,q,-q)$, $(-q,-q,q)\in V$. Therefore $p=q$ and $\vert V\vert=4.$
        
        \textbf{Case} $\vert V\vert=6$\textbf{.} There is another orbit of length $3$. Let $(q_{1},q_{2},q_{3})$, $(q_{2},q_{3},q_{1})$, $(q_{3},q_{1},q_{2}) \in V$. We say that it is $q$-orbit. $(x,y,z)\mapsto (x,-y,-z) \in SO_{3}(\Delta_{3})$ and, therefore, $(p_{1}, -p_{2}, -p_{3})\in V$.  
        Consider the following cases.

        If $(p_{1}, -p_{2}, -p_{3}) = (p_{1}, p_{2}, p_{3})$, then $p_{2}=p_{3}=0$ and $(-p_{1},0,0)$, $(0,-p_{1},0)$, $(0,0,-p_{1})$. Therefore $V = \{(p_{1},0,0)$, $(0,p_{1},0)$, $(0,0,p_{1}), (-p_{1},0,0)$, $(0,-p_{1},0)$, $(0,0,-p_{1})\}$ and $P$ is an octahedron. 

        If $(p_{1}, -p_{2}, -p_{3}) = (p_{2}, p_{3}, p_{1})$, then $(p_{1},-p_{1},-p_{1})\in V$ and $(p_{1},p_{1},p_{1})\in V$, thus, there is an orbit of length one, which is impossible in this case.

        If $(p_{1}, -p_{2}, -p_{3}) = (p_{3}, p_{1}, p_{2})$, then $(p_{1}, p_{1}, -p_{1}) \in V$ and $(-p_{1}, -p_{1}, -p_{1}) \in V$, thus, there is an orbit of length one, which is impossible in this case.

        The last case is $(p_{1}, -p_{2}, -p_{3}) \in q$-orbit. Using previous cases we assume that $(-p_{1}, -p_{2}, p_{3})$, $(-p_{1}, p_{2}, -p_{3}) \in q$-orbit and, what is more, that $q$-orbit consist of $\{(p_{1}, -p_{2}, -p_{3})$, $(-p_{1}, -p_{2}, p_{3})$, $(-p_{1}, p_{2}, -p_{3})\}$. Therefore for some choice of signs we have $(\pm p_{1}, \pm p_{1}, \pm p_{1})\in V$ and therefore for some choice of sign $\pm (p_{1}, p_{1}, p_{1})\in V$. Thus, there is an orbit of length one, which is impossible in this case.
    \end{proof}

    \begin{proof}[Equality case for convex bodies with tetrahedral symmetry]
        Equality can only be achieved if it holds in one of the Estimates \ref{Est1} or \ref{Est2}. We will consider the corresponding cases.

        If equality is achieved in Estimate \ref{Est1}, then it is also achieved in Estimate \ref{hexagon_est}. Therefore, $Pr_{(1,1,1)^{\perp}}K$ is a triangle. Also we have $\vert K\vert=\frac{1}{3}\vert \frac{1}{3a} (1,1,1)\vert \vert Pr_{(1,1,1)^{\perp}}K\vert$. Therefore, it follows from the proof of Lemma \ref{zang_lemma} that the Steiner symmetrization of $K$ with respect to $(1,1,1)^{\perp}$ is the convex hull of at most five points. Thus, $K$ is a polyhedron with at most five vertices. By Lemma \ref{56points}, $K$ is a tetrahedron.

        If equality is achieved in Estimate \ref{Est2}, then it is also achieved in Estimate \ref{square_est}. Therefore $Pr_{(0,0,1)^{\perp}}K$ is a parallelogram.  Also we have $\vert K\vert=\frac{1}{3} \vert \frac{1}{b}(0,0,1)\vert \vert Pr_{(0,0,1)^{\perp}}K\vert$. Therefore, it follows from the proof of Lemma \ref{zang_lemma} that the Steiner symmetrization of $K$ with respect to $(0,0,1)^{\perp}$ is the convex hull of at most six points. Thus, $K$ is a polyhedron with at most six vertices. By Lemma \ref{56points}, $K$ is a tetrahedron or an octahedron. Since the equality cannot be achieved in the case of an octahedron, the proof is complete.
 \end{proof}

\end{document}